\documentclass[11pt]{amsart}
\usepackage{amsthm}
\usepackage{amssymb}
\usepackage[margin=1.2in]{geometry}
\usepackage{mathrsfs}
\usepackage{bm}

\numberwithin{equation}{section}

\theoremstyle{plain}
\newtheorem{theorem}{Theorem}[section]
\newtheorem{lemma}[theorem]{Lemma}

\newtheorem{coro}[theorem]{Corollary}

\theoremstyle{definition}

\newtheorem{claim}[theorem]{Claim}

\newtheorem{remark}[theorem]{Remark}

\usepackage{color}

\newcommand{\C}{{\mathbb{C}}}
\newcommand{\N}{{\mathbb{N}}}
\newcommand{\R}{{\mathbb{R}}}

\newcommand{\Z}{{\mathbb{Z}}}

\newcommand{\calJ}{{\mathcal{J}}}

\newcommand{\calLP}{{\mathcal{LP}}}

\newcommand{\atr}{{\mathrm{atr}}}

\renewcommand{\Im}{{\mathrm{Im}\,}}
\newcommand{\Leb}{{\mathrm{Leb}}}
\newcommand{\LP}{{\mathrm{LP}}}

\renewcommand{\Re}{{\mathrm{Re}\,}}
\newcommand{\SO}{{\mathrm{SO}}}
\newcommand{\SL}{{\mathrm{SL}}}
\newcommand{\tr}{{\mathrm{Tr\,}}}

\DeclareMathOperator*{\wlim}{{w-lim}}

\newenvironment{claimproof}[1][Proof of Claim]{\noindent \underline{#1.} }{\hfill$\diamondsuit$}

\begin{document}

\title[Limit-Periodic Jacobi Matrices]{Thin Spectra and Singular Continuous Spectral Measures for Limit-Periodic Jacobi Matrices}

\author[D.\ Damanik]{David Damanik}
\address{Department of Mathematics, Rice University, Houston, TX~77005, USA}
\email{damanik@rice.edu}

\author[J.\ Fillman]{Jake Fillman}
\address{Department of Mathematics, Texas State University, San Marcos, TX 78666, USA}
\email{fillman@txstate.edu}

\author[C.\ Wang]{Chunyi Wang}
\address{Department of Mathematics, Rice University, Houston, TX~77005, USA}
\email{cw72@rice.edu}

\maketitle

\begin{abstract}
This paper investigates the spectral properties of Jacobi matrices with limit-periodic coefficients. We show that for a residual set of such matrices, the spectrum is a Cantor set of zero Lebesgue measure, and the spectral measures are purely singular continuous. For a dense set of limit-periodic Jacobi matrices we can strengthen the result and show that the spectrum is a Cantor set of zero lower box counting dimension, and hence in particular of zero Hausdorff dimension, while still retaining the singular continuity of the spectral type. We also show how results of this nature can be established by fixing the off-diagonal coefficients and varying only the diagonal coefficients, and, in a more restricted version, by fixing the diagonal coefficients to be zero and varying only the off-diagonal coefficients. We apply these results to produce examples of weighted Laplacians on the multidimensional integer lattice having purely singular continuous spectral type and zero-dimensional spectrum.
\end{abstract}

\setcounter{tocdepth}{1}
\tableofcontents

\section{Introduction and Results}

\subsection{General Background}
The present paper is concerned with Jacobi matrices $J = J_{a,b}$, acting in $\ell^2(\Z)$ via
\begin{equation} \label{eq:intro:jacobidef} [Ju]_n = a_{n-1}u_{n-1} + b_n u_n + a_n u_{n+1},\end{equation}
with $a_n>0$, $b_n \in \R$, $n \in \Z$. We shall consider bounded Jacobi matrices, for which
\begin{equation} \label{eq:intro:jacnormdef} \langle J \rangle := \max\left\{ \|a\|_\infty, \|a^{-1}\|_\infty, \|b\|_\infty \right\} \leq M<\infty.\end{equation}
Given $M>1$, we denote $\calJ_M$ the set of all Jacobi matrices with $\langle J \rangle \leq M$ and $\calJ$ the set of all bounded Jacobi matrices with $ \langle J \rangle <\infty$.

Jacobi matrices play a fundamental role in the spectral theory of bounded self-adjoint operators. In the spectral subspace generated by a designated vector $\psi$, the spectral theorem allows one to conjugate the given bounded self-adjoint operator to the multiplication by the independent variable in $L^2(\R,d\mu_\psi)$, where $\mu_\psi$ is the spectral measure associated with $\psi$. This measure will be compactly supported and the Gram-Schmidt procedure then turns the monomials into an orthonormal basis of $L^2(\R,d\mu_\psi)$, with respect to which the multiplication operator is then represented by a symmetric tridiagonal matrix with strictly positive off-diagonal elements -- a Jacobi matrix.\footnote{This discussion sweeps some minor subtleties under the rug.} As a result, one-sided infinite Jacobi matrices are the canonical building blocks in the spectral theory of bounded self-adjoint operators, and hence their spectral analysis is naturally important.

The passage to two-sided infinite Jacobi matrices is natural whenever the coefficients have a natural extension to the left half-line, and moreover the spectral analysis of the operator makes use of this extension. Such a scenario arises, for example, in the case where the coefficient sequences are obtained by sampling along the orbits of an \emph{invertible} discrete-time dynamical system.

A very prominent special case of such an invertible discrete-time dynamical system is given by a minimal translation on a compact abelian group. If the sampling functions are continuous, the resulting sequences are almost periodic in the sense of Bohr or Bochner. Among almost periodic sequences, special emphasis has been paid to quasi-periodic and limit-periodic sequences, which correspond to minimal shifts on tori of finite dimension and procyclic groups, respectively. We also note that the intersection of these two subclasses consists precisely of the periodic sequences.

Jacobi matrices with quasi-periodic coefficient sequences have been studied in numerous recent papers; compare, for example, \cite{AJM17, HanJit2017Adv, HanYanZha2020Jdam, K17, M14, T14, T20, TV17} and references therein. On the other hand, Jacobi matrices with limit-periodic sequences have not yet been studied in similar depth and it is therefore our goal in this paper to advance the understanding of such operators.

Another useful perspective on Jacobi matrices arises from viewing them as generalizations of discrete Schr\"odinger operators, which are Jacobi matrices whose off-diagonal terms are $a_n = 1$ for every $n$. The latter operators are of interest due to their central role in quantum mechanics, admittedly in a very simple setting -- a discretized one-dimensional position space. The existing literature on discrete Schr\"odinger operators with quasi-periodic or limit-periodic potentials (i.e., diagonal coefficients) is vast. Early work on limit-periodic operators was begun in the 1980s by authors including  Avron-Simon \cite{AS81}, Chulaevskii \cite{Chul81, Chul84}, Egorova \cite{egorova}, Molchanov--Chulaevskii \cite{MC84}, Moser \cite{M81}, Pastur--Tkachenko \cite{PT1, PT2}, and P\"oschel \cite{P83}. There has been a recent uptick in activity, ushered in by Avila \cite{A09}, which has led to the realization that new phenomena are possible, such as the genericity of purely singular continuous spectrum \cite{A09, DG11} in the space of all limit-periodic potentials.  Another new phenomenon established in a recent paper is the Lipschitz continuity of the integrated density of states in P\"oschel's examples \cite{DF2018}.  One can also study higher dimensional limit-periodic and quasi-periodic operators, which has been done in works of Karpeshina--Lee \cite{KarLee2007JAM, KarLee2009OTAA, KarLee2013JAM},  Karpeshina--Parnovski--Shterenberg \cite{KarParSht, KarParSht2}, and others. For surveys, book treatments, and additional references on quasi- and limit-periodic Schr\"odinger operators, we refer the reader to \cite{CFKS, D17, DF20, DF21a, DF21b, MJ17, PasturFigotinBook} and references therein. 

Existing results for discrete one-dimensional Schr\"odinger operators with quasi-periodic or limit-periodic potentials generally form the basis for subsequent extensions to Jacobi matrices (i.e., more general off-diagonal terms). The present paper follows the same philosophy and explores ways to extend some of the existing results for discrete Schr\"odinger operators with limit-periodic potentials \cite{A09, DG10, DG11} to suitable results for Jacobi matrices with limit-periodic coefficients.

\subsection{Setting and Main Results}

For $p \in \N$, we say that $J$ of the form \eqref{eq:intro:jacobidef} is $p$-periodic if $a_{n+p}\equiv a_n$ and $b_{n+p} \equiv b_n$ and \emph{limit-periodic} if there are periodic $J^{(1)},J^{(2)},\ldots$ such that $J^{(n)} \to J$ in the operator norm as $n \to\infty$.
For each $M> 1$, let $\calLP_M$ denote the set of limit-periodic elements of $\calJ_M$, that is, the set of limit-periodic $J$ for which $\langle J \rangle \leq M$, where $\langle \cdot \rangle$ is as in \eqref{eq:intro:jacnormdef}.
For each $M$, $\calLP_M$ is a complete metric space in the metric given by the operator norm.

Our first main result concerns the genericity of zero-measure spectrum within this class of operators.

\begin{theorem} \label{t:genericZMspec}
Let $M > 1$. For generic $J \in \calLP_M$, $\sigma(J)$ is a Cantor set of zero Lebesgue measure, and the spectral type of $J$ is purely singular continuous.
\end{theorem}

\begin{remark}
(a) As usual, a statement is said to hold for a generic element of a complete metric space if the set of elements with the property in question is residual, that is, it contains a dense $G_\delta$ subset of the underlying space.
\\[1mm]
(b) A compact subset of $\R$ is called a \emph{Cantor set} if it has empty interior and no isolated points.
\\[1mm]
(c) The spectral type of $J$ is purely singular continuous if for every $\psi \in \ell^2(\Z)$, the spectral measure $\mu_\psi$ is singular continuous, that is, it is supported by some set of zero Lebesgue measure and it gives no weight to individual points.
\end{remark}

By a quantitative refinement, one can sharpen the measure estimate to a dimensional estimate, at the price of exchanging a residual set for a dense set of operators.

\begin{theorem} \label{t:denseZHDimspec} For a dense set of $J \in \calLP_M$, $\sigma(J)$ is a Cantor set of zero lower box counting dimension, and the spectral type is again purely singular continuous.
\end{theorem}

\begin{remark} \label{rem:zeroboxtozerohd}
Recall that for a bounded subset of $\R$, its Hausdorff dimension is bounded from above by its lower box counting dimension. Thus, in the setting of Theorem~\ref{t:denseZHDimspec} (as well as that of Theorems~\ref{t:denseZHDimspecFixedOffD} and \ref{t:denseZHDimspecFixedDiag} below), the spectrum in particular has zero Hausdorff dimension and zero Lebesgue measure.
\end{remark}

We will in fact derive a slightly stronger statement by fixing the off-diagonal and varying only the diagonal entries. Let $\LP \subset \ell^\infty(\Z)$ denote the set of limit-periodic sequences, that is, the elements of the closure of the set of periodic sequences in $\ell^\infty(\Z)$.

\begin{theorem} \label{t:genericZMspecFixedOffD}
Let $a$ be periodic with $a_n>0$ for every $n$. For generic $b \in \LP$, $\sigma(J_{a,b})$ is a Cantor set of zero Lebesgue measure, and the spectral type of $J_{a,b}$ is purely singular continuous.
\end{theorem}

\begin{theorem} \label{t:denseZHDimspecFixedOffD}
Let $a$ be periodic with $a_n>0$ for every $n$. For a dense set of $b \in \LP$, $\sigma(J_{a,b})$ is a Cantor set of zero lower box counting dimension, and the spectral type of $J_{a,b}$ is purely singular continuous.
\end{theorem}

It is straightforward to see that Theorem~\ref{t:genericZMspecFixedOffD} implies Theorem~\ref{t:genericZMspec} and that Theorem~\ref{t:denseZHDimspecFixedOffD} implies Theorem~\ref{t:denseZHDimspec}, so we focus on the proof of Theorems~\ref{t:genericZMspecFixedOffD} and \ref{t:denseZHDimspecFixedOffD}. The proof revolves around a perturbative construction pioneered by Avila \cite{A09} and later generalized by Damanik-Fillman-Lukic \cite{DFL2017}.

Given the construction of Theorem~\ref{t:genericZMspecFixedOffD}, it is natural to wonder about the complementary perturbative argument, namely, what one can say for fixed $b$ and varying $a$. To that end, we consider off-diagonal $J$, that is, $J$ such that $b \equiv 0$. We will write $J_a = J_{a,0}$.
We again restrict to the bounded case for which one has
\begin{equation}\label{eq:langlearangledef} \langle J \rangle = \langle a \rangle:= \max\left\{ \|a\|_\infty, \|a^{-1}\|_\infty\right\} \leq M< \infty.\end{equation}
Writing $\R_+=(0,\infty)$, we shall abuse notation and identify a $p$-periodic $a \in \R_+^\Z$ with a corresponding element $a \in \R_+^p$ in the obvious manner. Given $M > 1$, let $\LP_M^+$ denote the set of positive limit-periodic sequences $a$ such that $\langle a \rangle \leq M$, that is, $M^{-1}\leq a_n \leq M$ for all $n$.

\begin{theorem} \label{t:genericZMspecFixedDiag}
Let $M > 1$. For generic $a \in \LP_M^+$, $\sigma(J_a)$ is a Cantor set of zero Lebesgue measure, and the spectral type of $J_{a}$ is purely singular continuous.
\end{theorem}

\begin{theorem} \label{t:denseZHDimspecFixedDiag}
Let $M > 1$. For a dense set of $a \in \LP_M^+$, $\sigma(J_a)$ is a Cantor set of zero lower box counting dimension, and the spectral type of $J_{a}$ is purely singular continuous.
\end{theorem}

Let us remark that $J_{\lambda a} = \lambda J_a$, so $\sigma(J_{\lambda a,0}) = \lambda \sigma(J_{a,0})$ and the spectral type is $J_{\lambda a}$ is the same for all $\lambda \not= 0$; thus, one could add a coupling constant to the previous results for no extra effort.

For this result, the perturbative arguments of \cite{A09, DFL2017} must be modified to a multiplicative perturbation rather than an additive one, which leads to several additional complications. We are able to overcome these complications for $b\equiv 0$, but they are so far insurmountable for nonconstant periodic background. We regard it as an interesting open problem to establish full analogs of Theorems~\ref{t:genericZMspecFixedOffD} and \ref{t:denseZHDimspecFixedOffD}, that is, fixing an arbitrary periodic diagonal sequence and then proving the desired statements for a residual or dense set of off-diagonal sequences $a \in \LP_M^+$.\footnote{While this paper was in production, the paper \cite{EFGL2022JFA} appeared, giving a new approach to limit-periodic problems. We expect that the methods of that work can be fruitfully applied in the present context.}

To conclude, we note that the dimensional results (in particular the results about the lower box dimension) can yield examples of weighted Laplace operators on $\ell^2(\Z^d)$ with similar spectral properties. Concretely, given positive weights $\{w_{\bm n, \bm m}\}_{\bm n,\bm m \in \Z^d , \|\bm n - \bm m\|_1=1}$, the corresponding weighted Laplacian $L=L_w$ on $\ell^2(\Z^d)$ is defined by
\begin{equation}
[L_wu]_{\bm n} = \sum_{\|\bm m - \bm n\|_1=1} w_{\bm n,\bm m} u_{\bm m}. \end{equation}
We assume $M^{-1}\leq w_{\bm n,\bm m} \leq M$ and $w_{\bm n,\bm m}=w_{\bm n,\bm m}$ for all $\bm n,\bm m \in \Z^d$, which implies that $L_w$ is a bounded self-adjoint operator on $\ell^2(\Z^d)$.
Naturally, we say that $w$ is periodic if there is a full-rank subgroup $G\subseteq \Z^d$ such that $w_{\bm n + \bm g,\bm m + \bm g} = w_{\bm n,\bm m}$ for all $\bm n , \bm m \in \Z^d$, $\bm g \in G$. We say that $w$ is limit-periodic if it is a uniform limit of periodic weights.

\begin{theorem} \label{t:weightedL}
For any $d \in \N$, there exist limit-periodic weights $w$ such that  $L_w$ has purely singular continuous spectrum and $\sigma(L_w)$ is a Cantor set of zero lower box-counting dimension.
\end{theorem}

\subsection{Organization of the Paper}

The paper is organized as follows. In Section~\ref{sec:thin} we show how to construct periodic Jacobi matrices with thin spectra close to an arbitrary starting point. This is used to prove the statements about thin Cantor spectra (i.e., Theorems~\ref{t:genericZMspec}, \ref{t:denseZHDimspec}, and \ref{t:genericZMspecFixedOffD}--\ref{t:weightedL}) in Section~\ref{sec:cantor}.

The paper concludes an appendix containing results that are likely known to experts, but which we could not find in the form needed in this paper in the literature. Specifically, Appendix~\ref{sec:perIDS} works out a formula for the IDS of the density of states of a periodic Jacobi matrix (needed in Section~\ref{sec:thin}) in the spirit of Avila \cite{A09}.

\section{Spectral Estimates for Periodic Jacobi Matrices} \label{sec:thin}

In this section our primary goal is to produce thin spectra for periodic Jacobi matrices, as this is naturally a useful tool in constructing limit-periodic Jacobi matrices with thin spectra via suitable approximants of this kind. The way this is accomplished is via transfer matrix norm estimates. We therefore begin by recalling the concept of transfer matrices for general Jacobi matrices and then specialize to the case of periodic coefficients.

\medskip

Let $J = J_{a,b}$ be a bounded Jacobi matrix. If $u:\Z \to \C$ is a formal solution to $Ju = zu$ for a scalar $z \in \C$, one can recover the values of $u$ from any two consecutive values (say $u(0)$ and $u(1)$) via the transfer matrix formalism:
\begin{equation}\label{e.transfermatrix1}
\begin{bmatrix}u(n+1) \\ a_n u(n) \end{bmatrix} = A_z(n,m) \begin{bmatrix} u(m+1) \\ a_m u(m) \end{bmatrix}.
\end{equation}
Since the solution space is two-dimensional, the matrices $A_z(n,m)$ are uniquely determined by the mapping property expressed through \eqref{e.transfermatrix1} and they can be explicitly described as follows. In the case $n > m$, we have
\begin{equation}\label{e.transfermatrix2}
A_z(n,m) = \underbrace{\frac{1}{a_n} \begin{bmatrix} z-b_n & -1 \\ a_n^2 & 0 \end{bmatrix}}_{:= T_z(n)} \frac{1}{a_{n-1}} \begin{bmatrix} z - b_{n-1} & -1 \\ a_{n-1}^2 & 0 \end{bmatrix} \cdots \frac{1}{a_{m+1}} \begin{bmatrix} z - b_{m+1} & -1 \\ a_{m+1}^2 & 0 \end{bmatrix},
\end{equation}
and in the case $n < m$, we have $A_z(n,m) = A_z(m,n)^{-1}$. Naturally, we also have $A_z(n,n) = I$.

\medskip

Let us now assume that $J$ is $p$-periodic for $p \in \N$, that is, $a_{n+p} = a_n$ and $b_{n+p} = b_n$ for every $n \in \Z$. In this case we let $\Phi(z) = \Phi(z,p) := A_z(p,0)$ denote the \emph{monodromy matrix}. By the general theory of periodic Jacobi matrices (for which the reader may consult, e.g., \cite{simszego}), we can describe the spectrum of $J$ via the trace of the monodromy matrix:
\begin{equation}\label{e.periodicspectrum}
\sigma(J) = \{E \in \R : \tr \Phi(E) \in [-2,2]\}.
\end{equation}
As $\Phi$ is a polynomial of degree $p$, it follows from \eqref{e.periodicspectrum} that $\sigma(J)$ has at most $p-1$ gaps. In fact, according to standard parlance the set always has precisely $p-1$ gaps, although some of them may be closed. Namely, it can be shown that the set $\{E \in \R : \tr \Phi(E) \in (-2,2)\}$ has precisely $p$ connected components, separated by what are referred to as \emph{gaps}. A \emph{band} of $\sigma(J)$ is the closure of a connected component of $\{E \in \R : \tr \Phi(E) \in (-2,2)\}$, so there are precisely $p$ bands. The bands and gaps are naturally numbered from left to right, so one can speak of the $j$-th band ($1 \le j \le p$) and the $j$-th gap ($1 \le j \le p-1$). From this perspective, the $j$-th gap is closed if the right endpoint of the $j$-th band coincides with the left endpoint of the $(j+1)$-st band; otherwise one refers to the $j$-th gap as an open gap (which is then taken to be the open interval determined by these two points).

It follows from this discussion that at the endpoints of bands, we must have $\tr \Phi(E) = \pm 2$. This is then in particular true at every closed gap, and we note that $E$ corresponds to a closed gap if and only if $\Phi(E) = I$ or $\Phi(E)=-I$.

We also note that the limit
\begin{equation}\label{eq:LEdef} \lim_{n\to\infty} \frac{1}{n} \log\|A_z(n,0)\| = L(z) \end{equation}
exists and by Gelfand's formula is exactly given by
\begin{equation} \label{eq:LEtospr} L(z) = \frac{1}{p} \log \mathrm{spr}\, \Phi(z),\end{equation}
where $\mathrm{spr}$ denotes the spectral radius. This is called the \emph{Lyapunov exponent} and can be defined for more general classes of Jacobi matrices, but we only need the periodic case for the current manuscript. In view of \eqref{e.periodicspectrum}, we note that
\begin{equation} \label{eq:perspecLE}
\sigma(J) = \{E \in \R : L(E) = 0\}.
\end{equation}
Due to \eqref{eq:LEtospr}, we note that in this periodic setting, $L$ is continuous as a function of $E$ and also as a function of $(a,b) \in \R_+^p \times \R^p$. Note however the restriction to a fixed period with regard to continuity as a function of $(a,b)$.

We refer the reader to \cite{simszego} for more information on periodic Jacobi matrices. In addition, Appendix~\ref{sec:perIDS} of the present paper contains some results (specifically Theorem~\ref{t:per:ids:hilbschmidt} and Corollary~\ref{coro:bandLengthLinUB}) not covered in \cite{simszego} but needed in our work.

\medskip

Recall that our goal is to show that the spectrum of a suitable periodic Jacobi matrix is thin in the sense of Lebesgue measure. We first work out a way of establishing this in the off-diagonal case, as this case will require some novel ideas (the Schr\"odinger case had been studied before). So, we consider a $p$-periodic off-diagonal Jacobi matrix $J = J_a$. The transfer matrices now take the form
\begin{equation}\label{e.transfermatrix2od}
A_z(n,m) = \frac{1}{a_n} \begin{bmatrix} z & -1 \\ a_n^2 & 0 \end{bmatrix} \frac{1}{a_{n-1}} \begin{bmatrix} z & -1 \\ a_{n-1}^2 & 0 \end{bmatrix} \cdots \frac{1}{a_{m+1}} \begin{bmatrix} z & -1 \\ a_{m+1}^2 & 0 \end{bmatrix}
\end{equation}
when $n > m$. Recalling the Lyapunov exponent from \eqref{eq:LEdef}, we will sometimes write $L(z,a)$ to emphasize the dependence on $a$. Abusing notation, we parametrize $J_a$ by $a \in \R_+^p$.
Let us note that for any bounded\footnote{In particular, this statement requires only boundedness, not periodicity or limit-periodicity.} $a,\widetilde{a} \in \ell^\infty(\Z)$, one can check that
\begin{equation}\label{eq:JdistFromaDist}
\|J_a - J_{\widetilde{a}}\| \leq 2\|a-\widetilde{a}\|_\infty.
\end{equation}
This is useful because
\begin{equation} \label{eq:specHdDist}
d_{\rm Hd}(\sigma(A),\sigma(B)) \leq \|A-B\|
\end{equation}
whenever $A$ and $B$ are self-adjoint operators and $d_{\rm Hd}$ denotes the Hausdorff metric given by
\begin{equation} \label{eq:hdmetric:def}
d_{\rm Hd}(K_1,K_2)
:=
\inf\{ \varepsilon > 0 : F \subseteq B_\varepsilon(K)  \text{ and } K \subseteq B_\varepsilon(F) \},
\end{equation}
whenever $K_1$ and $K_2$ are nonempty compact subsets of $\R$, and in which $B_\varepsilon(X)$ denotes the open $\varepsilon$-neighborhood of the set $X \subseteq \R$. Working with and combining \eqref{eq:JdistFromaDist}, \eqref{eq:specHdDist}, and \eqref{eq:hdmetric:def} enables one to keep track of how the spectrum changes as we perturb the off-diagonals $a \in \ell^\infty(\Z)$.

 The first result shows that all gaps are generically open.

\begin{lemma} \label{lem:odjmGenericGapsOpen}
Let $p \in \N$ be given. For a dense open set of $a \in \R_+^p$, $\sigma(J_a)$ has $p-1$ open gaps.
\end{lemma}

\begin{proof}
Since the gap boundaries vary continuously in the parameters, the set of $a \in \R_+^p$ for which $\sigma(J_a)$ has $p-1$ open gaps is open, so we only need to explain why it is dense.

Suppose the $j$-th gap of $\sigma(J_a)$ is closed and equal to $\{E\}$, so that $\Phi(E) = \pm I$. For $\widetilde a_p \neq a_p$ close to $a_p$, and $\widetilde a_j = a_j$ for $1 \le j \le p-1$, consider the resulting $\widetilde a \in \R_+^p$. The associated monodromy is given by
\[\widetilde \Phi(E) = \pm \frac{1}{\widetilde a_p} \begin{bmatrix} E & -1 \\ \widetilde a_p^2 & 0 \end{bmatrix} \left[ \frac{1}{a_p} \begin{bmatrix} E & -1 \\ a_p^2 & 0 \end{bmatrix}\right]^{-1}
= \pm \frac{1}{a_p \widetilde a_p} \begin{bmatrix} a_p^2 & 0 \\ 0 & \widetilde a_p^2 \end{bmatrix}. \]
We have
$$
\mathrm{Tr} \, \widetilde \Phi(E) = \pm  \frac{a_p^2 + \widetilde a_p^2}{a_p \widetilde a_p} = \pm \left( \frac{a_p}{\widetilde a_p} + \frac{\widetilde a_p}{a_p}\right),
$$
which lies in $\R \setminus [-2,2]$, and hence $E$ belongs to a gap of $\sigma(J_{\widetilde a})$. Note that if the perturbation is small enough, we can be sure that this gap is indeed the $j$-th gap of $\sigma(J_{\widetilde a})$.

Denoting by $U_j\subseteq \R_+^p$ the set of those $a$ for which the $j$th gap is open, the preceding arguments imply $U_j$ is open and dense. Since the intersection of finitely many dense open sets is dense and open, $U_1 \cap \cdots \cap U_p$ is the desired dense open set.
\end{proof}

Because the perturbative construction has a multiplicative structure, the behavior of the spectrum near energy $0$ is particularly delicate. Accordingly, the following lemma will be useful. For $a \in \R_+^\Z$ periodic, write
\begin{equation} \label{eq:lambda0Def}
\lambda_0(a) = \min\{|E| : E \in \sigma(J_a)\}.
\end{equation}
\begin{lemma} \label{lem:ODJM0inspec}
Suppose $a \in \R_+^\Z$ is $p$-periodic. Then $0 \in \sigma(J_a)$ if and only if $p$ is odd or $p$ is even and
\begin{equation}\label{eq:0inspecCondition} \prod_{j=1}^{p/2} a_{2j} = \prod_{j=1}^{p/2} a_{2j-1}. \end{equation}
In particular, $\lambda_0(a)>0$ for a dense open set of $a \in \R_+^p$ whenever $p$ is even.
Moreover, $\sigma(J)$ is symmetric about zero. If $B_1,\ldots,B_p$ denote the bands of $\sigma(J)$ listed from left to right, then
\begin{equation} \label{eq:ODJMsymmetry}
B_{p+1-k}=-B_k
\end{equation}
for each $k$.
\end{lemma}

\begin{proof}
The one-step transfer matrices at energy $0$ are of the form
$$
T_0(n) = \begin{bmatrix} 0 & -a_n^{-1} \\ a_n & 0 \end{bmatrix},
$$
so the monodromy is obtained by multiplying such matrices across one period.

If $p$ is odd, then one can check that the associated monodromy $\Phi(0)$ is of the form
\[\Phi(0) = \begin{bmatrix} 0 & * \\ * & 0 \end{bmatrix},\]
so $0 \in \sigma(J_a)$. If $p$ is even, then, by induction, one can see that
\[ \Phi(0) = \begin{bmatrix} (-1)^{p/2} r/s & 0 \\ 0 & (-1)^{p/2} s/r \end{bmatrix},\]
where \[r = \prod_{j=1}^{p/2} a_{2j-1}, \quad s= \prod_{j=1}^{p/2} a_{2j}.\]
In particular, $|\tr \Phi(0)| \geq 2$, with equality if and only if $r=s$, proving the first statement. For a given even $p$, the set of $a \in \R_+^p$ for which \eqref{eq:0inspecCondition} fails is readily seen to be open and dense, proving the second statement.

For the remaining statements, note that the operator $U:\ell^2(\Z)\to\ell^2(\Z)$ given by $[U\psi](n)=(-1)^n\psi(n)$ is unitary and $UJU^{*}=-J$. Thus, $J$ and $-J$ are unitarily equivalent, so the remaining assertions hold.
\end{proof}

The next lemma provides the key construction of operators with thin spectra.

\begin{lemma}\label{lem:ODJMsmallspec}
Suppose $a\in \R_+^\Z$ is $p$-periodic. For all $\varepsilon>0$, there exist constants $c,C>0$ depending only on $\varepsilon$, $p$, and $\|a\|$ such that the following holds true. For any $N \in \N$ with $N \geq C$, there exists $\widetilde{a} \in \R_+^\Z$ of period $Np$ such that $\|a-\widetilde{a}\|_\infty < \varepsilon$ and $$\Leb(\sigma(J_{\widetilde{a}})) < e^{-cNp}.$$\end{lemma}

\begin{proof} The argument proceeds as in \cite[Lemma~3.1]{DFL2017}, but with a multiplicative perturbation rather than an additive one. Thus, the most challenging aspect of the proof is the proof of Claim~\ref{claim:GapsCoverAllSpectrum}, which has to account for the difference between behavior near zero and far from zero, a difficulty not present in previous cases.
The details follow.

Let $\varepsilon>0$ and a $p$-periodic element $a \in \R_+^\Z$ be given, and denote $M := \|a\|_\infty$.
By Lemma~\ref{lem:ODJM0inspec}, there exists a $2p$-periodic sequence $a' \in \R_+^\Z$ with $\|a'\|_\infty \leq M$ and
\begin{equation} \label{eq:thin:aa'dist} \|a - a' \|_\infty < \varepsilon/3
\end{equation} such that $0 \notin \sigma(J_{a'})$. In particular, $\lambda_0(a') > 0$ (recall $\lambda_0$ is defined in \eqref{eq:lambda0Def}).

To continue, choose two parameters $\delta_1,\delta_2>0$ small enough that
\begin{equation}\label{eq:thin:delta12choice}
\lambda_0(a')-2\delta_2 > 0, \quad
\frac{\delta_1+4\delta_2}{\lambda_0(a')-2\delta_2} < \frac{\varepsilon}{3M}, \quad \text{ and } \quad \delta_2 < \frac{\varepsilon}{3}.\end{equation}
Choose an integer $N'$ large enough that the length of each band of $\sigma(J_{a'})$ is less than $\delta_1$ when $J_{a'}$ is viewed as a $2N'p$-periodic operator (which is possible due to Corollary~\ref{coro:bandLengthLinUB}). According to  Lemma~\ref{lem:odjmGenericGapsOpen} and Lemma~\ref{lem:ODJM0inspec}, there exists a $2N'p$-periodic sequence $a'' \in \R_+^\Z$ with $\|a''\|_\infty \leq M$ and
\begin{equation} \label{eq:thin:a'a''dist}
\|a'' - a'\|_\infty < \delta_2
\end{equation}
 such that $0 \notin \sigma(J_{a''})$ and all of the $(2N' p - 1)$ gaps of $\sigma(J_{a''})$ are open. Therefore, by standard perturbative arguments (namely \eqref{eq:JdistFromaDist}, \eqref{eq:specHdDist}, and \eqref{eq:hdmetric:def}), each band $B$ of $\sigma(J_{a''})$ satisfies
\begin{equation} \label{eq:thin:bandUB1}\Leb(B) \leq \delta_1 + 4 \delta_2\end{equation}
and we have
\begin{equation} \label{eq:thin:lambda0LB1} \lambda_0(a'') \geq \lambda_0(a') - 2 \delta_2 > 0.\end{equation}

\begin{claim}\label{claim:GapsCoverAllSpectrum}
There is a finite set $\mathcal{F}$ of $2N' p$-periodic elements of $\R_+^\Z$ such that
\begin{equation}
\label{eq:thin:a''ajdist}
\|\hat{a} - a''\|_\infty < \frac{\varepsilon}{3} \text{ for all } \hat{a} \in \mathcal{F}
\end{equation}
 and
\begin{equation} \label{eq:GapsCoverAllSpectrum}
\bigcap_{\hat{a} \in \mathcal{F}} \sigma(J_{\hat{a}}) = \emptyset.
\end{equation}
\end{claim}

\begin{claimproof}
The finite family of coefficient sequences that we produce in this proof will be obtained by multiplying $a''$ by scalars that are close to $1$.

By \eqref{eq:ODJMsymmetry} from Lemma~\ref{lem:ODJM0inspec}, it suffices to work on the part of spectrum contained in the positive half-line, and we may enumerate the bands\footnote{Strictly speaking, to be fully consistent with our earlier usage, we should enumerate these as $B_{k}$ with $N'p<k\leq 2N'p$, but the notation we used here is a little cleaner and should not cause confusion.} of $\sigma(J_{a''})$ contained in the right half-line from left-to-right as $B_1,B_2,\ldots,B_{N'p}$ and write $B_n = [E_n^-,E_n^+]$ for each $n$. Putting together \eqref{eq:thin:delta12choice}, \eqref{eq:thin:bandUB1}, and \eqref{eq:thin:lambda0LB1}, we see
\[\frac{|B_n|}{E_n^-} \leq \frac{\delta_1+4\delta_2}{\lambda_0(a')-2\delta_2} < \frac{\varepsilon}{3M}.\]
Since $|B_n| = E_n^+-E_n^-$, this naturally leads to
\begin{equation} \label{eq:thin:EnpmEst1} \frac{E_n^+}{E_n^-} < 1+\frac{\varepsilon}{3M}.\end{equation}
Define
\[ \Lambda = \max_{1 \le n \le N'p} \frac{E_n^+}{E_n^-}, \]
which, in view of \eqref{eq:thin:EnpmEst1}, satisfies $\Lambda < 1+ \frac{\varepsilon}{3M}$.
Choose $K \in \N$ large enough that
\begin{equation} \label{eq:thin:Cchoice}\Lambda^{1/K} <  \min_{1 \le n < N'p} \frac{E_{n+1}^-}{E_n^+}, \end{equation}
and put
\[a^{(k)} = \Lambda^{k/K} a'', \quad k \in [-K,K]\cap \Z.\]
We claim that $\mathcal{F} = \{a^{(k)} : k \in \Z, \ |k|\leq K\}$ is the desired finite set.
First, note $\Lambda>1$ so $1-\Lambda^{-1} < \Lambda-1$, which implies that for each $k \in \Z$ with $|k|\leq K$, one has
\[ \|a^{(k)} - a''\|_\infty = |\Lambda^{k/K}-1|\|a''\|_\infty \leq M (\Lambda-1) < \frac{\varepsilon}{3},\]
proving \eqref{eq:thin:a''ajdist}.

It remains to show \eqref{eq:GapsCoverAllSpectrum}, that is, every $E \in \R$ belongs to the resolvent set of some $J_{a^{(k)}}$. If $E \notin \sigma(J_{a''}) = \sigma(J_{a^{(0)}})$, then we are already done, so assume $E \in \sigma(J_{a^{(0)}}) = \sigma(J_{a''})$. As before, without loss, we may take $E>0$ and thus we consider $E \in B_n = [E_n^-,E_n^+]$. By our choice of $\Lambda$, $E = \Lambda^t E_n^-$ for some $0 \le t \le 1/2$ or $E = \Lambda^t E_n^+$ for some $-1/2\le t \le 0$. The cases are similar, so we focus on the case $E = \Lambda^t E_n^-$, $0 \le t \le 1/2$.

 Choosing the smallest integer $k$ for which $k/K$ is \emph{strictly} larger than $t$, we claim $\Lambda^{-k/K}E$ belongs to the resolvent set of $J_{a^{(k)}}$. If $n=1$, this is immediate from the definition of $k$. Otherwise, for $n\geq2$, we claim $\Lambda^{-k/K}E$ belongs to the gap $(E_{n-1}^+,E_n^-)$. For, if not, then
\[\Lambda^{-k/K}E \leq E_{n-1}^+ < E_n^- = \Lambda^{-t}E,\]
leading to
\[\Lambda^{1/K} \geq \Lambda^{\frac{k}{K}-t} \geq \frac{E_n^-}{E_{n-1}^+}, \]
contrary to \eqref{eq:thin:Cchoice}. Thus, $\Lambda^{-k/K}E$ belongs to the gap, which in turn implies that $E$ belongs to the resolvent set of $J_{a^{(k)}}$, as desired.
Putting $\mathcal{F} = \{a^{(k)} : |k|\leq K\}$ completes the proof of the claim.
\end{claimproof}

\bigskip

Take $\ell = 2K+1 = \#\mathcal{F}$ where $\mathcal{F}$ is the set constructed in Claim~\ref{claim:GapsCoverAllSpectrum}, and relabel the elements of $\mathcal{F}$ as $a^{(1)},\ldots,a^{(\ell)}$. Recall that $L(E,a)$ refers to the Lyapunov exponent at energy $E$ associated with the operator $J_a$. Since $L(E,a)$ is a continuous function of $E \in \R$ that is positive away from the spectrum of $J_a$ and goes to $\infty$ as $|E|\to\infty$, Claim~\ref{claim:GapsCoverAllSpectrum} gives
\begin{equation} \label{eq:thin:etaDef}
\eta := \min_{E \in \mathbb{R}} \max_{1 \leq j \leq \ell} L(E, a^{(j)}) > 0.
\end{equation}
We now have all the ingredients needed to define $\widetilde{a}$. Choose a large $N$ and take $\widetilde{N}$ to be the largest integer satisfying $ 2\ell(\widetilde{N} + 1) N'\leq  N$. We construct the $\widetilde{p} := Np$-periodic operator by concatenating the $a^{(j)}$'s, repeating each of them $\widetilde{N} + 1$ times, and filling in the remainder with copies of $a$.

 More precisely, denote $s_j = j(\widetilde{N} + 1)(2 N' p)$ and define a $\widetilde p$-periodic sequence $\widetilde{a}$ by
\[
\widetilde{a}_n =  \begin{cases}
  a^{(j)}_n &\mbox{if } s_{j-1} < n \leq s_j, \\
 a_n & \mbox{if }  s_\ell < n \leq \widetilde{p}.
\end{cases}
\]
Putting together \eqref{eq:thin:aa'dist}, \eqref{eq:thin:delta12choice}, \eqref{eq:thin:a'a''dist}, and \eqref{eq:thin:a''ajdist}, we see that
\[\|\widetilde{a}-a\|_\infty < \varepsilon.\]
Let $D$ denote the discriminant associated to $J_{\widetilde{a}}$ and let $E\in \R$ be given with $|D(E)|<2$. By \eqref{eq:thin:etaDef}, $L(E, a^{(j)}) \geq \eta$ for some $j$. We can use this to estimate transfer matrices across subintervals of $[s_{j-1}, s_j)$. To set notation, write $A^a_E(n,m)$ for the transfer matrix corresponding to the coefficient sequence $a$ (compare \eqref{e.transfermatrix1}--\eqref{e.transfermatrix2}) and define the following for $m \in \Z$:
\begin{align*}
X_{E}(m) & = A_E^{\widetilde{a}} (m + 2\widetilde{N}N'p, m)\\
\Phi_{E}(m) &= A_E^{\widetilde{a}} (m +\widetilde{p}, m).
\end{align*}
Thus, $X_{E}(m)$ transfers across a subinterval of length $2N'\widetilde{N}p$ beginning at $m$ and $\Phi_{E}(m)$ is the monodromy of $\widetilde{a}$ originating at $m$.
By $\widetilde{p}$-periodicity of $\widetilde{a}$, the reader will note that
\begin{align}
\nonumber
\Phi_E(m) & = X_E(m+\widetilde{p})^{-1}\Phi_E(m+2\widetilde{N}N'p)X_E(m) \\
\label{eq:thin:PhiXconj}
&= X_E(m)^{-1}\Phi_E(m+2\widetilde{N}N'p)X_E(m)\end{align}
for any $m \in \Z$.

Using the definition of $\widetilde{a}$ and $2N'p$-periodicity of $a^{(j)}$, we obtain
\begin{align*}
\|X_{E}(s_{j-1})\| &=
\|A_E^{\widetilde{a}} (s_{j-1}  + 2\widetilde{N}N'p , s_{j-1}) \| \\
& = \|A_E^{a^{(j)}} ( s_{j-1}  + 2\widetilde{N}N'p, s_{j-1}) \| \\[2mm]
& = \|A_E^{a^{(j)}} (  2\widetilde{N}N'p ,0) \| \\[2mm]
& = \left\|(A_E^{a^{(j)}} (2N'p, 0) )^{\widetilde{N}}\right\|.
\end{align*}
Naturally, the matrix appearing in the last line is (a power of) a monodromy matrix for the $2N'p$-periodic operator $J_{a^{(j)}}$, so invoking \eqref{eq:LEtospr}, and \eqref{eq:thin:etaDef} gives us
\begin{align*}
\|X_{E}(s_{j-1}) \|
& \geq \mathrm{spr}\left((A_E^{a^{(j)}} (2N'p,0) )^{\widetilde{N}}\right) \\[2mm]
& = e^{(2N'p L(E,a^{(j)})) \widetilde{N}}\\[2mm]
& \geq e^{2N'\widetilde{N}p\eta}.
\end{align*}
By maximality of $\widetilde{N}$, we have $2\ell(\widetilde{N}+2)N'>N$, leading to $\widetilde{N}N' > \frac{N}{2\ell} - 2N'$. Taking $N$ large enough, we can ensure that $2N'<N/4\ell$; putting this into the result of the previous calculation, we arrive at
\begin{equation}
\| X_{E}(s_{j-1})\|
\geq e^{\widetilde{p}\eta/(4\ell)}. \end{equation}

Since $|D(E)|<2$, the corresponding monodromy matrices are conjugate to a rotation $R_\theta$ (cf.\ \eqref{eq:Mconjdef}). In particular, there exists $M_E(s_{j - 1}) \in \mathrm{SL}(2, \mathbb{R})$ such that
\[
M_E(s_{j - 1}) \Phi_E(s_{j-1}) \left[M_E(s_{j - 1})\right]^{-1} = R_\theta.
\]
Applying \eqref{eq:thin:PhiXconj} with $m=s_{j-1}$, we note
\begin{align*}
X_E(s_{j-1})\Phi_E(s_{j-1})[X_E(s_{j-1})]^{-1}
&= \Phi_E(s_{j-1}+2\widetilde{N}N'p)
\end{align*}
Since $\Phi_E(s_{j-1}+2\widetilde{N}N'p)$ is conjugated to a rotation by $M_E(s_{j-1}+2\widetilde{N}N'p)$, it follows that $\Phi_E(s_{j-1})$ is conjugate to $R_\theta$ by
both $M_E(s_{j - 1})$ and $M_E(s_{j-1}+2\widetilde{N}N'p)X_E(s_{j-1})$. By uniqueness of conjugacies modulo $\SO(2, \mathbb{R})$ (cf.\ \eqref{eq:conj:uniqueUpToRot} and the surrounding discussion), there exists $Q \in \mathrm{SO}(2, \mathbb{R})$ such that
\[
M_E(s_{j-1}+2\widetilde{N}N'p)X_E(s_{j-1}) = QM_E(s_{j - 1}).
\]

Since $\|M\| = \|M^{-1}\|$ for $M \in \SL(2,\R)$ and $\|Q\| = 1$, we have
\begin{align*}
\max \{\|M_E (s_{j - 1})\|, \|M_E(s_{j} -2N'p)\|\} & \geq \|X_E(s_{j-1})\|^{1/2} \\
& \geq e^{\widetilde{p}\eta /(8\ell)}.
\end{align*}
According to Theorem~\ref{t:per:ids:hilbschmidt}, this gives
\[
\frac{dk}{dE} \geq \frac{C}{\widetilde{p}} e^{\widetilde{p}\eta /(4\ell)}
\]
for all $E$ with $|D(E)|<2$. The DOSM gives weight $1/\widetilde{p}$ to each band of $\sigma(J_{\widetilde{a}})$, so for any band $B \subseteq \sigma(J_{\widetilde{a}})$, we have
\[
|B| \lesssim  e^{-\widetilde{p}\eta/(4\ell)}.
\]
Therefore, for sufficiently large $N$ (hence sufficiently large $\widetilde{p}$), there exists a constant $c$ such that
\[
\Leb(\sigma(J_{\widetilde{a}})) \lesssim \widetilde{p} e^{-\widetilde{p}\eta/(4\ell)} \leq e^{-c \widetilde{p}},
\]
as desired.
\end{proof}

Lemma~\ref{lem:ODJMsmallspec} also holds for $J_{a, b}$ with fixed periodic off-diagonals $a$ and changing diagonals $b$. The statement follows:

\begin{lemma} \label{lem:DJMsmallspec}
Suppose $a,b$ are both $p$-periodic sequences. For all $\varepsilon>0$, there exist constants $c,C>0$ depending only on $\varepsilon$, $p$, $\|a\|$ and $\|b\|$ such that the following holds true. For any $N \in \N$ with $N \geq C$, there exists $\widetilde{b}$ of period $Np$ such that $\|b-\widetilde{b}\|_\infty < \varepsilon$ and $$\Leb(\sigma(J_{a, \widetilde{b}})) < e^{-cNp}.$$\end{lemma}

The proof of Lemma \ref{lem:DJMsmallspec} is almost the same as Lemma \ref{lem:ODJMsmallspec} except for Claim~\ref{claim:GapsCoverAllSpectrum}. One constructs $\{b^{(1)},\ldots, b^{(\ell)}\}$ by shifting $b \mapsto b + c$ and hence the proof is similar to Avila \cite{A09} and D.--F.--Lukic \cite{DFL2017}.

\begin{remark}
The present paper is focused on the spectrum and spectral type, but the methods we describe have other applications. For instance, by following the arguments of Kr\"uger--Gan \cite{KrugerGan2011MathNach}, one can use Lemma~\ref{lem:ODJMsmallspec} to show that for a dense set of limit-periodic positive sequences $a \in \R_+^\Z$, the integrated density of states of the off-diagonal Jacobi matrix $J_a$ has no positive H\"{o}lder modulus of continuity.
\end{remark}

\section{Singular Continuous Cantor Spectrum} \label{sec:cantor}

We are now ready to prove our main results. We begin with the proofs of Theorems~\ref{t:genericZMspecFixedDiag} and \ref{t:denseZHDimspecFixedDiag}.

\begin{proof}[Proof of Theorem~\ref{t:genericZMspecFixedDiag}]
For $M > 1, \delta>0$, let $U_{M,\delta}$ denote the set of all $a \in \LP_M^+$ for which $\Leb(\sigma(J_a)) < \delta$.

\begin{claim}
For all $M$ and $\delta$, $U_{M,\delta}$ is a dense and open subset of $\LP_M^+$.
\end{claim}

\begin{claimproof}
We first show $U_{M,\delta}$ is dense. Since the periodic elements are dense in $\LP_M^+$, it suffices to show that, for any periodic $a$ such that $\|a\|_\infty < M$, every neighborhood of $a$ intersects $U_{M,\delta}$. So, let $a \in \LP_M^+$ be periodic of period $p$ with $\|a\|_\infty < M$. According to Lemma~\ref{lem:ODJMsmallspec}, for any $\varepsilon > 0$, there exists $\widetilde{a}$ with period $Np$ and $\|a - \widetilde{a}\|_\infty <\varepsilon$ satisfying
\[
\Leb (\sigma(J_{\widetilde{a}})) < \delta,
\]
proving that $U_{M,\delta}$ is dense.

Let us show $U_{M,\delta}$ is open. Given $a \in U_{M,\delta}$, cover $\sigma(J_a)$ by finitely many intervals $I_1,\ldots,I_k$ such that $\sum |I_j| < \delta$. Then, if $\delta'$ is small enough, any $a'$ with $\|a-a'\|_\infty < \delta'$ will satisfy $\Leb(\sigma(J_{a'}))< \delta$, and hence $U_{M,\delta}$ is open.
\end{claimproof}

Thus, with
\[U_M := \bigcap_{n=1}^\infty U_{M,\frac1n},\]
we have $\Leb(\sigma(J_a)) = 0$ for every $a \in U_M$, and $U_M$ is a dense $G_\delta$.

The proof that the spectrum is generically continuous relies on suitable versions of the Gordon lemma \cite{Gordon1976UMN} for Jacobi matrices, which have been worked out in \cite{HanJit2017Adv, HanYanZha2020Jdam, Seifert2019}. To that end, define $\mathcal{O}_{M,k, p}$ to be
\[
\mathcal{O}_{M,k, p} := \{a \in \LP_M^+ : \max_{1 \leq n \leq p} |a_n - a_{n \pm p}| < k^{-p} \}.
\]
$\mathcal{O}_{M,k,p}$ is an open set, so
\[
\mathcal{O}_{M,N} := \bigcup_{k \geq N} \bigcup_{p \geq N} \mathcal{O}_{M,k, p}
\]
is also open. Since all periodic elements of $\LP_M^+$ are contained in $\mathcal{O}_{M,N}$, $\mathcal{O}_{M,N}$ is dense as well.
Define $\mathcal{G}_M$ to be
\[\mathcal{G}_M = \bigcap_{N \geq 1} \mathcal{O}_{M,N}.\]
According to the Gordon lemma (e.g.\ \cite[Theorem~2.1]{Seifert2019}), for any $a \in \mathcal{G}_M$, the spectrum of $J_a$ is purely continuous.
Denote $\mathcal{C}$ to be
\[\mathcal{C}_M := \mathcal{G}_M \cap U_M.\]
Since $\mathcal{G}_M$ and $U_M$ are dense $G_\delta$ sets, $\mathcal{C}$ is generic in $\LP_M^+$. By construction, for any $a \in \mathcal{C}$, the spectrum $J_a$ is a set of zero Lebesgue measure and the spectral type is singular continuous. By general principles \cite{Pastur1980CMP}, the spectrum of $J_a$ has no isolated points and hence is a Cantor set.
\end{proof}

\begin{proof}[Proof of Theorem~\ref{t:denseZHDimspecFixedDiag}]
Let $M > 1$ be given and fix $a \in \LP^+_M$ and $\varepsilon > 0$. First, there exists a periodic $a^{(0)}$ with \begin{equation} \label{eq:odjmHD0:a0Est}\|a - a^{(0)}\| \leq \varepsilon/2.
\end{equation}
Let $p_0$ denote the period of $a^{(0)}$. To begin the construction, set $\varepsilon_1 = \varepsilon/4$, and use Lemma~\ref{lem:ODJMsmallspec} to choose $a^{(1)}$ of period $p_1 = N_1 p_0$ such that
$ \|a^{(0)}-a^{(1)}\|_\infty \leq \varepsilon_1$
and
$\mu_1 := \Leb(\sigma(J_{a^{(1)}})) \leq e^{-p_1^{1/2}}$.
Assume $\varepsilon_{n-1}$ and $a^{(n-1)}$ have been chosen so that
\[\|a^{(n-2)} - a^{(n-1)}\|_\infty \leq \varepsilon_{n-1},\]
$a^{(n-1)}$ has period $p_{n-1}$ and
\[\mu_{n-1} := \Leb(\sigma(J_{a^{(n-1)}})) \leq e^{-p_{n-1}^{1/2}}.\]
Define
\begin{equation} \label{eq:odjmHD0:epsNDef} \varepsilon_n = \min\left( \frac{\varepsilon_{n-1}}{2}, \frac{\mu_{n-1}}{8}, \frac{1}{2}n^{p_{n-1}} \right).\end{equation}
Applying Lemma~\ref{lem:ODJMsmallspec} again, we see that for a suitable large $N_n$, we may produce $a^{(n)}$ of period $p_n = N_n p_{n-1}$ such that
\begin{equation} \label{eq:odjmHD0:anEst} \|a^{(n)} - a^{(n-1)} \| \leq \varepsilon_n \end{equation}
and
\begin{equation}\label{e.sigmaJanmeas}
\Leb (\sigma(J_{a^{(n)}})) \leq e^{-p_n^{1/2}}.
\end{equation}
Applying \eqref{eq:odjmHD0:epsNDef} inductively (recall $\varepsilon_1 = \varepsilon/4$), we see that $\varepsilon_n \leq \varepsilon/2^{n+1}$ for all $n$. Thus, combining \eqref{eq:odjmHD0:a0Est} and \eqref{eq:odjmHD0:anEst}, we see that the limit
\[a^{(\infty)} = \lim_{n \to \infty} a^{(n)}\]
exists and satisfies
\[\|a^{(\infty)} - a^{(n)}\| \leq 2 \varepsilon_{n+1} \leq \varepsilon_n.\]
In particular, we note
\[\|a^{(\infty)} - a\| \leq \varepsilon\]
and (invoking \eqref{eq:odjmHD0:epsNDef} again)
\[\|a^{(\infty)} - a^{(n)}\| \leq \frac{\mu_n}{4}.\]
In view of \eqref{eq:JdistFromaDist} and \eqref{eq:specHdDist}, we see that $\sigma(J_a^{(\infty)})$ may be covered by the closed $\mu_n/2$-neighborhood of $\sigma(J_a^{(n)})$, which consists of $p_n$ intervals of length at most $2\mu_n \leq 2\exp(-p_n^{1/2})$
which suffices to show that the spectrum of $J_{a^{(\infty)}}$ has zero lower box-counting dimension (and hence zero Hausdorff dimension and zero Lebesgue measure as discussed in Remark~\ref{rem:zeroboxtozerohd}).

From \eqref{eq:odjmHD0:epsNDef}, we also see that
\[ \|a^{(n)} - a^{(\infty)}\|_\infty \leq n^{-p_n},\]
which means $J=J_{a^{(\infty)}}$ is a Jacobi matrix of Gordon type and hence has purely continuous spectrum as before. As already mentioned, the spectrum of $J$ has zero Lebesgue measure, so it can support no absolutely continuous measures and hence $J$ has purely singular continuous spectrum.
\end{proof}

\begin{proof}[Proof of Theorems~\ref{t:genericZMspecFixedOffD} and \ref{t:denseZHDimspecFixedOffD}]
These theorems follow from Lemma~\ref{lem:DJMsmallspec} in precisely the same fashion that Theorems~\ref{t:genericZMspecFixedDiag} and \ref{t:denseZHDimspecFixedDiag} followed from Lemma~\ref{lem:ODJMsmallspec}. 
\end{proof}

\begin{proof}[Proof of Theorems~\ref{t:genericZMspec} and \ref{t:denseZHDimspec}]
Theorems~\ref{t:genericZMspec} and \ref{t:denseZHDimspec} follow immediately from Theorems~\ref{t:genericZMspecFixedOffD} and \ref{t:denseZHDimspecFixedOffD}.
\end{proof}

We conclude by using the previous construction to prove Theorem~\ref{t:weightedL}, concerning weighted Laplacians in $\ell^2(\Z^d)$ with suitable limit-periodic hopping terms.

\begin{proof}[Proof of Theorem~\ref{t:weightedL}]
Apply Theorem~\ref{t:denseZHDimspecFixedDiag} to produce a limit-periodic sequence $\{a_n\}_{n\in \Z}$ such that the Jacobi matrix $J_{a,0}$ on $\ell^2(\Z)$ has purely singular continuous spectrum and satisfies
\[ \dim_{\rm B}^-(J_{a,0}) = 0, \]
and define $w_{\bm n, \bm m}$ by
\[
w_{\bm n, \bm n+ \bm e_j} = a_{n_j}, \quad
w_{\bm n, \bm n- \bm e_j} = a_{n_j-1}.
\]
Viewing $\ell^2(\Z^d)$ as  $\ell^2(\Z)^{\otimes d}$, the $d$-fold tensor product of $\ell^2(\Z)$ with itself, one has
\[
L_w = J_{a,0} \otimes I^{\otimes d-1} + I \otimes J_{a,0} \otimes I^{\otimes d-2} + \cdots + I^{\otimes d-1}\otimes J_{a,0},
\]
so one has
\[ \sigma(L_w)= \underbrace{\sigma(J_{a,0}) + \sigma(J_{a,0}) +\cdots + \sigma(J_{a,0})}_{d \text{ copies}}, \]
and the result follows immediately.
\end{proof}

\begin{appendix}

\section{The IDS for Periodic Jacobi Matrices} \label{sec:perIDS}

Our objective here is to derive suitable bounds on the density of states of periodic Jacobi operators from the general theory.
 If $J$ is a Jacobi matrix of period $p$, the density of states measure (DOSM) of $J$, denoted $dk$ can be defined by
\[\int \! f \, dk = \frac{1}{p}\sum_{n=0}^{p-1} \langle\delta_n, f(J)\delta_n\rangle, \quad z\in \C\setminus \R.\]
Equivalently, the DOSM is precisely the weak limit of normalized eigenvalue counting measures associated with restrictions of $J$ to finite boxes. That is, $dk = \wlim dk_N$ where \[\int \! f \, dk_N = \frac{1}{N} \sum_{j=1}^N f(E_j)\]
and $E_1<E_2 < \cdots < E_N$ are the eigenvalues of $J_N$, the restriction of $J$ to $[1,N]$ with Dirichlet boundary conditions (the reader can check that the eigenvalues of $J_N$ are simple). The \emph{integrated density of states} (IDS) is the accumulation function of the DOSM, that is,
\[ k(E) = \int \! \chi_{(-\infty,E]}\, dk. \]
 For a helpful introduction with proofs of basic facts, see \cite[Chapter~5]{simszego}.

Recall that any $A \in \SL(2,\R)$ induces a linear fractional transformation (LFT) on the sphere $\overline{\C} = \C \cup\{\infty\}$ via
$$
\begin{bmatrix}
a & b \\
c & d
\end{bmatrix} \cdot z
= \frac{az + b}{cz+d}.\
$$
We let $\C_+$ denote upper half-plane $\C_+ = \{ z \in \C : \Im z > 0\}$
and define
$$
R_\theta
=
\begin{bmatrix} \cos(\theta) & - \sin(\theta) \\ \sin(\theta) & \cos(\theta) \end{bmatrix}
$$
for $\theta \in \R$.

Let $J = J_{a,b}$ be a $p$-periodic Jacobi matrix. Recall that with the transfer matrices $A_z(n,m)$ defined in \eqref{e.transfermatrix1}--\eqref{e.transfermatrix2}, we can describe the spectrum of $J$ via \eqref{e.periodicspectrum}, that is,
\[
\sigma(J) = \{E \in \R : \tr \Phi(E) \in [-2,2]\},
\]
where the monodromy matrix $\Phi(z) = \Phi(z,p)$ is given by the transfer matrix over one period, $A_z(p,0)$.

Now, for $E$ with $D(E) := \tr \Phi(E) \in (-2,2)$, we have
$$
\tr  A_E(p+j,j) = D(E)\in (-2,2) \quad \text{ for each } j
$$
by periodicity of $J$ and cyclicity of the trace, so each of the monodromies $A_E(p+j,j)$ is conjugate to $R_\theta$, where $\theta$ satisfies
\begin{equation} \label{eq:floquetthetadef} 2 \cos(\theta) = D(E).\end{equation}
We shall denote the conjugacies by $M_E(j) = M_E(j;p)$, that is,
\begin{equation} \label{eq:Mconjdef}
M_E(j;p) A_E(p+j,j) M_E(j;p)^{-1} \in \SO(2,\R)
\end{equation}
for $j \in \Z$, $D(E) \in (-2,2)$.
Let us note that rotation matrices are uniquely characterized as the subgroup of $\SL(2,\R)$ that stabilizes $i \in \C_+$ under the action given by LFTs, and in general, $A \in \SL(2,\R)$ has trace in $(-2,2)$ if and only if the LFT action of $A$ on $\C_+$ has a unique fixed point. Thus, a conjugacy from an elliptic element $A \in \SL(2,\R)$ to a rotation is specified by choosing $M \in \SL(2,\R)$ whose associated LFT maps the unique fixed point of $A$ in $\C_+$ to $i$ and these conjugacies are unique modulo the stabilizer of $i$, that is, if $M_1$ and $M_2$ both satisfy \eqref{eq:Mconjdef}, one has
\begin{equation}\label{eq:conj:uniqueUpToRot}
M_1 = QM_2\end{equation} for a rotation $Q$.

From the general theory (e.g.\ \cite[Sections~5.3 and 5.4]{simszego}, see especially (5.3.34)), we know that the derivative of the integrated density of states can be related to $ |d\theta / dE| $ with $\theta$ from \eqref{eq:floquetthetadef}, so our first goal is to recover this derivative from the transfer matrices.
Let us introduce some notation.
Suppose $J$ is $p$-periodic and denote
$$
T_j = \frac{1}{a_j} \begin{bmatrix} E - b_j & - 1 \\ a_j^2 & 0 \end{bmatrix},
\quad
A_j = T_{j} \cdots T_1,
\quad
\Phi_j = A_{j-1} A_p A_{j-1}^{-1},
\quad
\text{for } j \geq 1,
$$
where we adopt the convention $A_0 = I$ in the $j = 1$ case of the final definition and suppress the dependence of all quantities on $E$ for notational simplicity. Finally, for $D(E) \in (-2,2)$, we denote $M_j = M_E(j;p)$.

\begin{theorem} \label{t:per:ids:hilbschmidt}
Let $J$ be  $p$-periodic with corresponding integrated density of states $k$. We have
\begin{equation} \label{eq:dkde:hilbschm}
\frac{dk}{dE}
 \geq \frac{C}{p} \sum_{j=1}^p \| M_j \|_2^2
\end{equation}
on $\sigma(J)$, where $\| M \|_2 = \sqrt{\tr(M^* M)}$ denotes the Hilbert--Schmidt norm of $M$ and $C$ is a constant that only depends on $\|a\|_\infty$. If $a$ is a constant sequence, then we have equality in \eqref{eq:dkde:hilbschm} with $C = \frac{1}{2(1+a_0^2)\pi}$.
\end{theorem}

Let us note that in general this is only an inequality, not an identity.
Let us note the following immediate corollary.

\begin{coro} \label{coro:bandLengthLinUB}
Let $J$ be $p$-periodic. Each band $B$ of $\sigma(J)$ satisfies
\[|B| \leq \frac{C}{p}\]
where $C$ is a constant that depends only on $\|a\|_\infty$.
\end{coro}

\begin{proof}
This is immediate from Theorem~\ref{t:per:ids:hilbschmidt}, the observation $\|M\|\geq \sqrt{2}$ for any $M \in \SL(2,\R)$ (by Cauchy--Schwarz), and the fact that the $dk$-measure of $B$ is precisely $1/p$.
\end{proof}

By way of motivation, suppose $\theta$ is a smooth function of $t$. By a direct calculation, one can verify that
$$
R_\theta^{-1} \frac{dR_\theta}{dt}
=
\begin{bmatrix} 0 & - d\theta/dt \\ d\theta/dt & 0 \end{bmatrix},
$$
which motivates us to define the \emph{anti-trace} of a $2 \times 2$ matrix by
$$
\atr \begin{bmatrix} a & b \\ c & d \end{bmatrix} = c - b.
$$
Naturally, for all matrices $A$ and $B$ and all scalars $\lambda$, one has
$$
\atr(A + \lambda B) = \atr(A) + \lambda \, \atr(B).
$$
However, $\atr$ is not cyclic; that is, one can have $\atr(AB) \neq \atr(BA)$.  However, we have the following weakened variant of cyclicity.

\begin{lemma} \label{l:atr:conj:inv}
If $R \in \SO(2,\R)$ and $A$ is any $2 \times 2$ matrix,
\begin{equation} \label{eq:atr:conj:inv}
\atr\left(R^{-1} A R \right)
=
\atr(A).
\end{equation}
\end{lemma}

\begin{proof}
Calculation.
\end{proof}

\begin{lemma} \label{l:smoothconj:to:rot}
Suppose $I \subset \R$ is an open interval and $\Phi: I \to \SL(2,\R)$ is a smooth map such that $\big| \tr(\Phi(t)) \big| < 2$ for all $t \in I$. Under these conditions, there exists a smooth choice of $M \in \SL(2,\R)$ such that
\begin{equation} \label{eq:smoothconj:to:rot}
M \Phi M^{-1} = R_\theta,
\end{equation}
where $2 \cos(\theta) = \tr(\Phi)$. Moreover, the angle $\theta$ can be chosen to be a smooth function of $t$; in this case, it satisfies
$$
\frac{d\theta}{dt} = \frac{1}{2} \atr \left( M \Phi^{-1} \frac{d\Phi}{dt} M^{-1} \right),
$$
\end{lemma}

\begin{proof}
To construct the conjugacy $M$, first notice that the unique fixed point $z=z(t)$ of $\Phi$ varies smoothly with $t$.  We then define
$$
M(t)
=
\big( \Im z(t) \big)^{-1/2} \begin{bmatrix} 1 & - \Re z(t) \\ 0 & \Im z(t) \end{bmatrix},
$$
By construction, the LFT corresponding to $M(t)$ maps $z(t)$ to $i$, so (the LFT corresponding to) $M \Phi M^{-1}$ fixes $i$, which implies $M \Phi M^{-1} \in \SO(2,\R)$. By cyclicity of the trace, $M \Phi M^{-1}$ must be of the claimed form. Differentiating the relation \eqref{eq:smoothconj:to:rot} yields
$$
\frac{dM}{dt} \Phi M^{-1} + M \frac{d\Phi}{dt} M^{-1} + M \Phi \frac{dM^{-1}}{dt}   = \frac{dR}{dt}
=
R \begin{bmatrix} 0 & - d\theta/dt \\ d\theta/dt & 0 \end{bmatrix}.
$$
Multiply on the left by $R^{-1}$ and simplify using \eqref{eq:smoothconj:to:rot} to obtain
\begin{equation} \label{eq:smoothconj:to:rot:step}
R^{-1} \frac{dM}{dt} M^{-1} R + M \Phi^{-1} \frac{d\Phi}{dt} M^{-1} + M\frac{d M^{-1}}{dt}
=
\begin{bmatrix} 0 & - d\theta/dt \\ d\theta/dt & 0 \end{bmatrix}.
\end{equation}
By \eqref{eq:atr:conj:inv}, linearity of the anti-trace, and the product rule,
\begin{align*}
\atr \left(R^{-1} \frac{dM}{dt} M^{-1} R + M \frac{dM^{-1}}{dt}\right)
& =
\atr\left(  \frac{dM}{dt} M^{-1} + M \frac{d M^{-1}}{dt} \right) \\
& =
\atr\left( \frac{d}{dt}(MM^{-1}) \right) \\
& =
0.
\end{align*}
Thus, \eqref{eq:smoothconj:to:rot} follows by taking the anti-trace of \eqref{eq:smoothconj:to:rot:step}.
\end{proof}

\begin{proof}[Proof of Theorem~\ref{t:per:ids:hilbschmidt}]
First, notice that $\|M_j\|_2$ does not depend on the choice of conjugacy, for any other conjugacy from $\Phi_j$ to a rotation must take the form $Q M_j$ for some $Q \in \SO(2,\R)$ (cf.\ \eqref{eq:conj:uniqueUpToRot} and the discussion surrounding it).
Since we may take $M_j$ to be given by the explicit formula
\[M_j = (\Im z_j)^{-1/2} \begin{bmatrix} 1 & -\Re z_j \\ 0 & \Im z_j  \end{bmatrix}, \]we see that
$$
\| M_j \|_2^2
=
\frac{1 + |z_j|^2}{\Im z_j},
$$
where $z_j \in \C_+$ is the unique fixed point of the action of $\Phi_j$ on the upper half-plane. Notice that $T_j z_j = z_{j+1}$ and hence $M_{j+1} T_j M_j^{-1} $ fixes $i$, so $M_{j+1} T_j M_j^{-1} = Q_j \in \SO(2,\R)$. One can compute
$$
T_j^{-1} \frac{dT_j}{dE} = \begin{bmatrix} 0 & 0 \\ -1 & 0 \end{bmatrix},
$$
Thus, differentiating $\Phi_1$ gives us
$$
\Phi_1^{-1} \frac{d\Phi_1}{dE}
=
\sum_{j=1}^{p} A_{j-1}^{-1} \begin{bmatrix} 0 & 0 \\ -1 & 0 \end{bmatrix} A_{j-1}.
$$
With $R_0 = I$ and $R_j = Q_{j} \cdots Q_1$ for $j \ge 1$, we have
\begin{equation} \label{eq:per:ids:hilbschmidt:deriv}
\Phi_1^{-1} \frac{d\Phi_1}{dE}
=
\sum_{j=1}^{p} M_1^{-1} R_{j-1}^{-1} M_{j} \begin{bmatrix} 0 & 0 \\ -1 & 0 \end{bmatrix} M_j^{-1} R_{j-1} M_1
\end{equation}
To find the rate of change of $\theta$ with respect to $E$, we apply Lemma~\ref{l:smoothconj:to:rot} and compute
\begin{align*}
\left| \frac{d\theta}{dE} \right|
& =
\left|\frac{1}{2} \atr \left( M_1 \Phi_1^{-1} \frac{d\Phi_1}{dE} M_1^{-1} \right) \right| \\
& =
\frac{1}{2} \left| \atr\left( \sum_{j=1}^{p} M_j \begin{bmatrix} 0 & 0 \\ -1 & 0 \end{bmatrix} M_j^{-1} \right) \right|  \\
& =
\frac{1}{2} \sum_{j=1}^{p} \frac{|z_j|^2}{\Im z_j }.
\end{align*}
The second line follows from \eqref{eq:per:ids:hilbschmidt:deriv} and Lemma~\ref{l:atr:conj:inv}, and the final line is a straightforward computation from the explicit form of $M_j$. Since
\[z_{j+1} = T_j z_j = \frac{(E-b_j)z_j-1}{a_j^2 z_j} = \frac{(E-b_j)}{a_j^2} - \frac{1}{a_j^2 z_j},\]
we note
$$
\Im z_{j+1}
=
\frac{\Im z_j}{a_j^2|z_j|^2}.
$$
If $a_j \equiv a_0$ is constant, this implies
\begin{align*}
\left| \frac{d\theta}{dE} \right|
=
\frac{1}{2}\sum_{j=1}^{p} \frac{|z_j|^2}{\Im z_j}
&=
\frac{1}{2} \sum_{j=1}^{p} \frac{1}{1+a_0^2}\frac{|z_j|^2}{\Im z_j} + \frac{a_0^2}{1+a_0^2}\frac{1}{a_0^2 \Im z_{j+1}} \\
&=
\frac{1}{2(1+a_0^2)} \sum_{j=1}^{p} \frac{|z_j|^2 + 1}{\Im z_j} \\
& =
\frac{1}{2(1+a_0^2)} \sum_{j=1}^p \|M_j\|_2^2.
\end{align*}
Otherwise, we still have the lower bound
\begin{align*}
\left| \frac{d\theta}{dE} \right|
=
\frac{1}{2}\sum_{j=1}^{p} \frac{|z_j|^2}{\Im z_j}
&=
\frac{1}{2} \sum_{j=1}^{p} \frac{1}{1+a_j^2}\frac{|z_j|^2}{\Im z_j} + \frac{a_j^2}{1+a_j^2}\frac{1}{a_j^2 \Im z_{j+1}} \\
&=
\frac{1}{2} \sum_{j=1}^{p} \frac{1}{1+a_j^2}\frac{|z_j|^2}{\Im z_j}+ \frac{1}{1+a_{j-1}^2}\frac{1}{\Im z_j} \\
& \geq
C \sum_{j=1}^p \|M_j\|_2^2,
\end{align*}
where $C>0$ only depends on $\|a\|_\infty$. In either case, the conclusion of the theorem follows from \cite[Eq.~(5.3.34)]{simszego}.
\end{proof}

\section*{Acknowledgements}

The authors are grateful to Artur Avila and Rui Han for helpful conversations. D.D.\ was supported in part by NSF grants DMS--1700131 and DMS--2054752, an Alexander von Humboldt Foundation research award, and Simons Fellowship $\# 669836$. J.F.\ was supported in part by NSF grant DMS--2213196 and Simons Collaboration Grant \#711663. 

\end{appendix}

\end{document}